\documentclass[12pt]{article}
\usepackage{booktabs}
\usepackage{caption}
\usepackage{mathrsfs}
\usepackage{amsmath}
\usepackage{amsfonts,amsthm,amssymb,mathrsfs,bbding}
\usepackage{graphics,multicol}
\usepackage{graphicx}
\usepackage{color}
\usepackage{enumerate}
\usepackage{caption}
\allowdisplaybreaks[4] 
\usepackage[
colorlinks,
anchorcolor=blue,
linkcolor=blue,
citecolor=blue]{hyperref}
\usepackage{extarrows}
\usepackage{cite}
\usepackage{latexsym,bm}
\usepackage{mathtools}
\evensidemargin=\oddsidemargin\topmargin=-1.5cm

\newtheorem{thm}{Theorem}[section]
\newtheorem{prob}{Problem}
\newtheorem{lem}[thm]{Lemma}

\theoremstyle{definition}

\addtocounter{section}{0}
\begin{document}
\title{The second largest eigenvalues of some Cayley graphs on alternating groups\footnote{This work is supported by the National Natural Science Foundation of China (Grant nos. 11531011, 11671344).}}
\author{{\small Xueyi Huang, \ \ Qiongxiang Huang\footnote{
Corresponding author.}\setcounter{footnote}{-1}\footnote{
\emph{E-mail address:} huangxymath@gmail.com (X. Huang), huangqx@xju.edu.cn (Q. Huang).}}\\[2mm]\scriptsize
College of Mathematics and Systems Science,
\scriptsize Xinjiang University, Urumqi, Xinjiang 830046, P. R. China}
\date{}
\maketitle
{\flushleft\large\bf Abstract}
Let $A_n$ denote the alternating group of degree $n$ with $n\geq 3$. The alternating group graph $AG_n$, extended alternating group graph $EAG_n$ and  complete alternating group graph $CAG_n$ are the Cayley graphs $\mathrm{Cay}(A_n,T_1)$, $\mathrm{Cay}(A_n,T_2)$ and $\mathrm{Cay}(A_n,T_3)$, respectively, where $T_1=\{(1,2,i),(1,i,2)\mid 3\leq i\leq n\}$, $T_2=\{(1,i,j),(1,j,i)\mid 2\leq i<j\leq n\}$ and $T_3=\{(i,j,k),(i,k,j)\mid 1\leq i<j<k\leq n\}$.  In this paper,  we determine the second largest eigenvalues of  $AG_n$, $EAG_n$ and  $CAG_n$. 

\begin{flushleft}
\textbf{Keywords:} Alternating group graph; Cayley graph; Second largest eigenvalue
\end{flushleft}
\textbf{AMS Classification:} 05C50

\section{Introduction}\label{s-1}
Let $G=(V(G),E(G))$ be a simple undirected graph of order $n$.  The \emph{adjacency matrix} of $G$, denoted by $A(G)$, is the $n\times n$ matrix with entries $a_{uv}=1$ if $\{u,v\}\in E(G)$ and $a_{uv}=0$ otherwise. The eigenvalues of $A(G)$ are denoted by $\lambda_1(G)\geq \lambda_2(G)\geq \cdots\geq \lambda_n(G)$, which are also called the \emph{eigenvalues} of  $G$. 

For $v\in V(G)$, we denote by $N(v)=\{u\in V(G)\mid \{u,v\}\in E(G)\}$ and $d(v)=|N(v)|$ the \emph{neighborhood}  and \emph{degree} of $v$, respectively. Let $D(G)=\mathrm{diag}(d(v)\mid v\in V(G))$ denote the diagonal degree  matrix of $G$. The \emph{Laplacian matrix} of $G$ is defined as $L(G)=D(G)-A(G)$, which is positive semi-definite and always has $0$ as its eigenvalue. So the eigenvalues of $L(G)$ can be arranged as $\mu_1(G)\geq \mu_2(G)\geq \cdots\geq \mu_{n-1}(G)\geq \mu_n(G)=0$. Furthermore, $\mu_{n-1}(G)>0$ if and only if $G$ is connected, and for connected non-complete  graphs we know that $0<\mu_{n-1}(G)\leq\kappa(G)\leq\kappa'(G)\leq \delta(G)$, where $\kappa(G)$, $\kappa'(G)$ and $\delta(G)$ denote the vertex connectivity, edge connectivity and minimum degree of $G$, respectively  (cf. \cite{Cvetkovic}, Corollary 7.4.6). For this reason, $\mu_{n-1}(G)$ is  called the \emph{algebraic connectivity} of $G$. Clearly, if $G$ is $r$-regular then $\mu_{n-1}(G)=r-\lambda_2(G)$.

For a finite group $\Gamma$, and a subset $T\subseteq \Gamma\setminus\{e\}$ ($e$ is the identity element of $\Gamma$) such that $T=T^{-1}$, the \emph{Cayley graph} $\mathrm{Cay}(\Gamma,T)$ on $\Gamma$ with respect to $T$ is defined as the undirected graph with vertex set $\Gamma$ and  edge set $\{\{\gamma,t\gamma \}\mid\gamma\in\Gamma,t\in T\}$. Clearly,  $\mathrm{Cay}(\Gamma,T)$ is a $|T|$-regular graph which is connected if and only if $T$ is a generating  subset of  $\Gamma$.

Let  $S_n$ be the symmetric group of degree $n\geq 3$ (for $\sigma,\tau\in S_n$, the product $\sigma\tau$ means that we first apply permutation $\sigma$ and then apply permutation $\tau$, and $\tau_i$ represents the image of $i$ under the permuation $\tau$), and let $T$ be a set of transpositions in $S_n$. The \emph{transposition graph} $\mathrm{Tra}(T)$ of $T$  is defined as the graph with vertex set $\{1,2,\ldots,n\}$ and with an edge connecting two vertices $i$ and $j$ if and only if $(i,j)\in T$. It is known that $T$ can generate $S_n$ if and only if  $\mathrm{Tra}(T)$ is connected \cite{Godsil1}. Around 1992,  Aldous \cite{Aldous} (see also \cite{Friedman,Cesi}) conjectured that if $\mathrm{Tra}(T)$ is connected then the second largest eigenvalue of $\mathrm{Cay}(S_n,T)$ is equal to $|T|-\mu_{n-1}(\mathrm{Tra}(T))$, that is, $\mathrm{Cay}(S_n,T)$ and $\mathrm{Tra}(T)$ share the same algebraic connectivity. Before 1992, Flatto et al.\cite{Flatto}  and Diaconis and Shahshahani \cite{Diaconis} had confirmed the conjecture for  $\mathrm{Tra}(T)=K_{1,n-1}$ and $\mathrm{Tra}(T)=K_{n}$, respectively.  In 2000, Friedman \cite{Friedman} proved that if $\mathrm{Tra}(T)$ is a tree ($|T|=n-1$) then the second largest eigenvalue of $\mathrm{Cay}(S_n,T)$ is at least $n-2$, and  exactly equal to $n-2$ if and only if $\mathrm{Tra}(T)=K_{1,n-1}$.  In 2010, Cesi \cite{Cesi} confirmed Aldous' conjecture when $\mathrm{Tra}(T)$ is a complete multipartite graph.  Almost at the same time, Caputo et al. \cite{Caputo} completely solved the conjecture for all connected $\mathrm{Tra}(T)$. For other Cayley graphs on symmetric groups, Cesi \cite{Cesi1} proved that the second largest eigenvalue of the pancake graph $\mathcal{P}_n=\mathrm{Cay}(S_n,S)$  ($S=\{(1,i)(2,i-1)\cdots\mid 2\leq i\leq n\}$) is equal to $n-2$; and very recently, Chung and Tobin \cite{Chung} determined the second largest eigenvalues of the reversal graph $R_n=\mathrm{Cay}(S_n,S)$ ($S=\{(i,j)(i+1,j-1)\cdots\mid1\leq i<j\leq n\}$) and a family of graphs that generalize the pancake graph. Their method is very imaginative, which depends on finding enough equitable partitions of $R_n$ and decomposing the edge set $E(R_n)$ to that of $\mathcal{P}_{n}$ and $n$ copies of $R_{n-1}$.

Now we consider the  Cayley graphs on alternating group $A_n$ ($n\geq 3$). Let $T_1=\{(1,2,i),(1,i,2)\mid 3\leq i\leq n\}$, $T_2=\{(1,i,j),(1,j,i)\mid 2\leq i<j\leq n\}$ and $T_3=\{(i,j,k),(i,k,j)\mid 1\leq i<j<k\leq n\}$. The \emph{alternating group graph} $AG_n$, \emph{extended alternating group graph} $EAG_n$ and  \emph{complete alternating group graph} $CAG_n$ are defined as the Cayley graphs $\mathrm{Cay}(A_n,T_1)$, $\mathrm{Cay}(A_n,T_2)$ and $\mathrm{Cay}(A_n,T_3)$, respectively. Since $T_1\subseteq T_2\subseteq T_3$,  $AG_n$ is a spanning subgraph of $EAG_n$ while $EAG_n$ is also a spanning subgraph of $CAG_n$, and all these graphs are connected because $T_1$ can generate  $A_n$ (cf. \cite{Suzuki}, p. 298).  In \cite{Jwo}, Jwo et al. introduced the alternating group graph $AG_n$ as an interconnection network topology for computing systems, and they also showed that $AG_n$ is $2$-transitive, Hamiltonian and of diameter $\lfloor 3(n-2)/2\rfloor$.  Following this paper, $AG_n$ was extensively studied over decades. For example, Chang and Yang \cite{Chang,Chang1}  investigated the panconnectivity, Hamiltonian-connectivity and  fault-tolerant Hamiltonicity of $AG_n$. Zhou \cite{Zhou} determined the full automorphism group of $AG_n$. In addition, the complete alternating group graph $CAG_n$ can be viewed as one of the two isomorphic connected components of the $(n-3)$-point fixing graph $\mathcal{F}(n,n-3)$ defined in \cite{Cheng}, which has only integral eigenvalues (cf. \cite{Cheng}, Corollary 1.2). Recently, the full automorphism group of $CAG_n$ was determined by Huang and Huang \cite{Huang}. 

Inspired by the work of Chung and Tobin \cite{Chung},  in the present paper, we determine the second largest eigenvalues of $AG_n$, $EAG_n$ and $CAG_n$, which are equal to $2n-6$ (for $n\geq 4$), $n^2-5n+5$ and $\frac{1}{3}n(n-2)(n-4)$, respectively. 

The paper is organized as follows. In Section \ref{s-2}, we obtain the second largest eigenvalue of $AG_n$ by induction on $n$.  In Section \ref{s-3}, we focus on determining the second largest eigenvalue of $CAG_n$. In order to achieve this goal by induction on $n$, we decompose the edge set of $CAG_n$ into that of $EAG_n$ and $n$ copies of $CAG_{n-1}$, and use $\lambda_2(EAG_n)+\lambda_2(CAG_{n-1})$ to bound $\lambda_2(CAG_n)$. Thus we have to determine $\lambda_2(EAG_n)$, and for this reason, we  also decompose the edge set of $EAG_n$ into that of $AG_n$ and $n$ copies of $EAG_{n-1}$, and use $\lambda_2(AG_n)+\lambda_2(EAG_{n-1})$ to bound $\lambda_2(EAG_n)$ as $\lambda_2(AG_n)$ has been determined in Section \ref{s-2}.

\section{The second largest eigenvalue of $AG_n$}\label{s-2}

In this section, we focus on determining the second largest eigenvalue of $AG_n$. In order to obtain enough information about the eigenvalues of $AG_n$, we  need a well known result on equitable partition. 

Let  $G$ be a graph of order $n$, and let $\Pi:V(G)=V_1\cup V_2\cup\cdots\cup V_k$ be a vertex partition of $G$.  The \emph{characteristic matrix} $\chi_\Pi$ of $\Pi$  is the $n\times k$ matrix whose columns are the character vectors of $V_1,\ldots,V_k$. The partition $\Pi$ is called an  \emph{equitable partition} of $G$ if, for any $v\in V_i$, $|N(v)\cap V_j|=b_{ij}$ is a constant only dependent on $i,j$ ($1\le i,j\le k$). The matrix $B_\Pi=(b_{ij})_{k\times k}$ is called the \emph{divisor matrix} of $G$ with respect to $\Pi$. 

\begin{lem}[\cite{Godsil1}]\label{lem-2-1}
Let $G$ be a graph with adjacency matrix $A(G)$, and let $\Pi:V(G)=V_1\cup V_2\cup\cdots\cup V_k$ be an equitable partition of $G$ with divisor matrix $B_{\Pi}$. Then each eigenvalue of $B_\Pi$ is also an eigenvalue of $A(G)$. Furthermore, $A(G)$ has the following two kinds of eigenvectors:
\begin{enumerate}[(i)]
\vspace{-0.2cm}
\item the eigenvectors in the column space of $\chi_\Pi$, and the corresponding eigenvalues coincide
with the eigenvalues of $B_\Pi$;
\vspace{-0.2cm}
\item the eigenvectors orthogonal to the columns of $\chi_\Pi$, i.e., those eigenvectors that sum to zero on each block $V_i$ for $1\leq i\leq k$.
\end{enumerate}
\end{lem}

Recall that $AG_n=\mathrm{Cay}(A_n,T_1)$, where $T_1=\{(1,2,i),(1,i,2)\mid 3\leq i\leq n\}$. Clearly, $AG_n$ is $(2n-4)$-regular. If $n=3$, then $AG_n=K_3$, which has $-1$ as its second largest eigenvalue. Now suppose $n\geq 4$. For each fixed $i$ ($1\leq i\leq n$), we define
\begin{align}\label{equ-1}
\begin{aligned}
&X(i)=\{\tau\in A_n\mid\tau_n=i\},\\
&Y(i)=\{\tau\in A_n\mid\tau_1=i\},\\
&Z(i)=\{\tau\in A_n\mid\tau_2=i\},\\
&W(i)=\{\tau\in A_n\mid\tau_1\neq i,\tau_2\neq i,\tau_n\neq i\}.
\end{aligned}
\end{align}
Now we verify that $\Pi: A_n=X(i)\cup Y(i) \cup Z(i) \cup W(i)$ is exactly an equitable partition of $AG_n$. For each $\tau\in X(i)$, we have $N(\tau)\cap X(i)=\{(1,2,k)\tau,(1,k,2)\tau\mid 3\leq k\leq n-1\}$, $N(\tau)\cap Y(i)=\{(1,n,2)\tau\}$, $N(\tau)\cap Z(i)=\{(1,2,n)\tau\}$ and $N(\tau)\cap W(i)=\emptyset$. This implies that each $\tau\in X(i)$ has exactly $2n-6$ neighbors in $X(i)$, one neighbor in $Y(i)$, one neighbor in $Z(i)$ and no neighbors in $W(i)$.  Similarly, for each $\tau\in Y(i)$, we have $N(\tau)\cap X(i)=\{(1,2,n)\tau\}$, $N(\tau)\cap Y(i)=\emptyset$, $N(\tau)\cap Z(i)=\{(1,k,2)\tau\mid 3\leq k\leq n\}$ and $N(\tau)\cap W(i)=\{(1,2,k)\tau\mid 3\leq k\leq n-1\}$, and for each $\tau\in Z(i)$, we have $N(\tau)\cap X(i)=\{(1,n,2)\tau\}$, $N(\tau)\cap Y(i)=\{(1,2,k)\tau\mid 3\leq k\leq n\}$, $N(\tau)\cap Z(i)=\emptyset$ and $N(\tau)\cap W(i)=\{(1,k,2)\tau\mid 3\leq k\leq n-1\}$. For each $\tau\in W(i)$, we know that  $\tau_\ell=i$ for some $3\leq \ell\leq n-1$ because $\tau_1,\tau_2,\tau_n\neq i$. Then $N(\tau)\cap X(i)=\emptyset$, $N(\tau)\cap Y(i)=\{(1,\ell,2)\tau\}$, $N(\tau)\cap Z(i)=\{(1,2,\ell)\tau\}$ and $N(\tau)\cap W(i)=\{(1,2,k)\tau,(1,k,2)\tau\mid 3\leq k\leq n, k\neq \ell\}$. Therefore, the divisor matrix of $AG_n$ with respect to $\Pi$ is equal to 
$$B_\Pi=
\left[
\begin{matrix}
2n-6 &1& 1& 0\\
1&0&n-2&n-3\\
1&n-2&0&n-3\\
0&1&1&2n-6
\end{matrix}
\right].
$$
By simple computation, we know that $B_\Pi$ has eigenvalues $2n-4$, $2n-6$, $n-4$ and $2-n$. Thus, by Lemma \ref{lem-2-1}, we have
\begin{lem}\label{lem-2-2}
Let $AG_n$ be the alternating group graph with $n\geq 4$. Then $AG_n$ has eigenvalues $2n-4$, $2n-6$, $n-4$ and $2-n$. Moreover, if $f$ is an eigenvector of $AG_n$ corresponding to any other eigenvalue, then for each $i\in\{1,2,\ldots,n\}$, we have $$\sum_{x\in X(i)}f(x)=\sum_{y\in Y(i)}f(y)=\sum_{z\in Z(i)}f(z)=\sum_{w\in W(i)}f(w)=0,$$ where $X(i)$, $Y(i)$, $Z(i)$ and $W(i)$ are defined in (\ref{equ-1}). 
\end{lem}
Now we are in a position to prove the main result of this section.
\begin{thm}\label{thm-1}
Let $AG_n$ be the alternating group graph with $n\geq 4$. Then the two largest eigenvalues of $AG_n$ are $\lambda_1(AG_n)=2n-4$ and $\lambda_2(AG_n)=2n-6$. 
\end{thm}
\begin{proof}
Clearly, the largest eigenvalue of $AG_n$ is $\lambda_1(AG_n)=2n-4$ because $AG_n$ is $(2n-4)$-regular. We shall prove  $\lambda_2(AG_n)=2n-6$ by induction on $n$. If $n=4$,  the result follows since all the distinct eigenvalues of $AG_4$ are $4$, $2$, $0$ and $-2$.  Now suppose $n>4$, and assume that the result holds for $n-1$, that is, $\lambda_2(AG_{n-1})=2(n-1)-6=2n-8$. 

Let $\lambda$ be an eigenvalue of $AG_n$ other than $2n-4$, $2n-6$, $n-4$ and $2-n$. It suffices to show $\lambda<2n-6$. Take any fixed eigenvector $f$ of $AG_n$ corresponding to $\lambda$. We have the following claim.

\vspace{0.2cm}
\noindent\textbf{Claim 1.} There exists an $i$ such that 
$$
2\sum_{x\in X(i)}f(x)^2\geq \sum_{y\in Y(i)}f(y)^2+\sum_{z\in Z(i)}f(z)^2
$$
and
$$
\sum_{x\in X(i)}f(x)^2>0.
$$
\begin{proof}
Notice that the vertex set of $AG_n$ can be partitioned in the following ways: $A_n=X(1)\cup X(2)\cup\cdots\cup X(n)=Y(1)\cup Y(2)\cup \cdots\cup Y(n)=Z(1)\cup Z(2)\cup \cdots\cup Z(n)$. Thus we have 
$$
\sum_{j=1}^n\sum_{x\in X(j)}f(x)^2=\sum_{j=1}^n\sum_{y\in Y(j)}f(y)^2=\sum_{j=1}^n\sum_{z\in Z(j)}f(z)^2>0,
$$ 
and hence
$$
\sum_{j=1}^n2\sum_{x\in X(j)}f(x)^2=\sum_{j=1}^n\left(\sum_{y\in Y(j)}f(y)^2+\sum_{z\in Z(j)}f(z)^2\right)>0.
$$ 
Therefore, there exists an index $i$ such that 
$$
2\sum_{x\in X(i)}f(x)^2\geq \sum_{y\in Y(i)}f(y)^2+\sum_{z\in Z(i)}f(z)^2.
$$ 
Let $I$ be the set of indices $i$ satisfying this inequality. Then we claim that there exists some $i\in I$ satisfying 
$$
\sum_{x\in X(i)}f(x)^2>0
$$
because $f\neq 0$.
\end{proof}
Now fix an $i$ satisfying the inequalities in Claim 1, and consider an arbitrary vertex $x\in X(i)$. As noted above, $x$ has $2n-4$ neighbors in $AG_n$, in which $2n-6$ neighbors in $X(i)$, one neighbor in $Y(i)$ and one neighbor in $Z(i)$. Observe that the induced subgraph of $AG_n$ on $X(i)$ is isomorphic to $AG_{n-1}=\mathrm{Cay}(A_{n-1},T_1')$, where $T_1'=\{(1,2,k),(1,k,2)\mid{3\leq k\leq n-1}\}$. In fact, if $i=n$ then $AG_n[X(i)]$ is obviously isomorphic to $AG_{n-1}$ since $\tau_n=n$ for each $\tau \in X(n)$; if $i<n$, we define 
\begin{eqnarray*}
\phi:X(i)&\longrightarrow &A_{n-1}\\
\tau=\left(\begin{matrix}
1&2&\cdots&l&\cdots& n-1&n\\
\tau_1&\tau_2&\cdots&n&\cdots& \tau_{n-1}&i\\
\end{matrix}\right)&\longmapsto&
\left(\begin{matrix}
1&2&\cdots&l&\cdots& n-1\\
\tau_1&\tau_2&\cdots&i&\cdots& \tau_{n-1}\\
\end{matrix}\right)
\end{eqnarray*}
It is easy to check that $\phi$ is one-to-one and onto. For any two distinct vertices  $\tau=\left(\begin{matrix}
1&2&\cdots&l&\cdots& n-1&n\\
\tau_1&\tau_2&\cdots&n&\cdots& \tau_{n-1}&i\\
\end{matrix}\right)$ and $\tau'=\left(\begin{matrix}
1&2&\cdots&m&\cdots& n-1&n\\
\tau_1'&\tau_2'&\cdots&n&\cdots& \tau_{n-1}'&i\\
\end{matrix}\right)$ of $X(i)$, we have 
\begin{eqnarray*}
\tau'\tau^{-1}
&=&\left(\begin{matrix}
1&2&\cdots&m&\cdots& n-1&n\\
\tau_1'&\tau_2'&\cdots&n&\cdots& \tau_{n-1}'&i\\
\end{matrix}\right)\left(\begin{matrix}
\tau_1&\tau_2&\cdots&n&\cdots& \tau_{n-1}&i\\
1&2&\cdots&l&\cdots& n-1&n\\
\end{matrix}\right)\\
&=&\left(\begin{matrix}
1&2&\cdots&m&\cdots& n-1&n\\
*&*&\cdots&l&\cdots& *&n\\
\end{matrix}\right)\\
&=&\left(\begin{matrix}
1&2&\cdots&m&\cdots& n-1\\
*&*&\cdots&l&\cdots& *\\
\end{matrix}\right)\\
&=&\left(\begin{matrix}
1&2&\cdots&m&\cdots& n-1\\
\tau_1'&\tau_2'&\cdots&i&\cdots& \tau_{n-1}'\\
\end{matrix}\right)\left(\begin{matrix}
\tau_1&\tau_2&\cdots&i&\cdots& \tau_{n-1}\\
1&2&\cdots&l&\cdots& n-1\\
\end{matrix}\right)\\
&=&\phi(\tau')\phi(\tau)^{-1}.
\end{eqnarray*}
Then $\{\tau,\tau'\}\in E(G[X_i])$ if and only if $\tau'\tau^{-1}\in T_1'$, which is the case if and only if $\phi(\tau')\phi(\tau)^{-1}\in T_1'$, which is the case if and only if $\{\phi(\tau),\phi(\tau')\}\in E(AG_{n-1})$. Thus $\phi$ is exactly an isomorphism from $G[X(i)]$ to $AG_{n-1}$. Also note that the edges between $X(i)$ and $Y(i)$ (resp. $X(i)$ and $Z(i)$) form a matching. Let $x'$ (resp. $x''$) be the unique neighbor of $x$ in $Y(i)$ (resp. $Z(i)$). By the eigenvalue-eigenvector equation, we have 
$$\lambda f(x)=\sum_{y\in N(x)\cap X(i)}f(y)+f(x')+f(x'')$$
and further,
$$\lambda f(x)^2=\sum_{y\in N(x)\cap X(i)}f(x)f(y)+f(x)(f(x')+f(x'')).$$
Summing both sides over $x\in X(i)$, we obtain 
$$\lambda\sum_{x\in X(i)} f(x)^2=\sum_{x\in X(i)}\sum_{y\in N(x)\cap X(i)}f(x)f(y)+\sum_{x\in X(i)}f(x)(f(x')+f(x'')),$$
which gives that 
\begin{equation}\label{equ-2}
\lambda=\frac{\sum_{x\in X(i)}\sum_{y\in N(x)\cap X(i)}f(x)f(y)}{\sum_{x\in X(i)} f(x)^2}+\frac{\sum_{x\in X(i)}f(x)(f(x')+f(x''))}{\sum_{x\in X(i)} f(x)^2}.
\end{equation}
Here we use the fact $\sum_{x\in X(i)} f(x)^2>0$ by Claim 1. Now we  shall find upper bounds for the two terms in (\ref{equ-2}). Let $G_1=AG_n[X(i)]$. Then $G_1\cong AG_{n-1}$ as mentioned above. Set $g=f|_{X(i)}$. We have $g\perp \mathbf{1}$ since $\sum_{x\in X(i)}f(x)=0$ according to Lemma \ref{lem-2-2}, where $\mathbf{1}$ is the all ones vector, which is also the eigenvector of $\lambda_1(G_1)$. Thus 
\begin{eqnarray*}
\frac{\sum_{x\in X(i)}\sum_{y\in N(x)\cap X(i)}f(x)f(y)}{\sum_{x\in X(i)} f(x)^2}&=&\frac{g^TA(G_1)g}{g^Tg}\\
&\leq &\max_{h\perp \mathbf{1}}\frac{h^TA(G_1)h}{h^Th}\\
&=&\lambda_2(G_1)\\
&=&\lambda_2(AG_{n-1})\\
&=&2n-8.
\end{eqnarray*}
Next consider the second term in (\ref{equ-2}). Recall that the edges between $X(i)$ and $Y(i)$, and $X(i)$ and $Z(i)$ form two matchings. Thus when $x$ range over the vertices of $X(i)$, $x'$ and $x''$ will range over the vertices of $Y(i)$ and $Z(i)$, respectively. Then 
\begin{eqnarray*}
\frac{\sum_{x\in X(i)}f(x)(f(x')+f(x''))}{\sum_{x\in X(i)} f(x)^2}&\leq&\frac{\sqrt{(\sum_{x\in X(i)} f(x)^2)(\sum_{x\in X(i)} (f(x')+f(x''))^2)}}{\sum_{x\in X(i)}f(x)^2}\\
&=&\sqrt{\frac{\sum_{x\in X(i)}(f(x')+f(x''))^2}{\sum_{x\in X(i)}f(x)^2}}\\
&\leq&\sqrt{\frac{\sum_{x\in X(i)}2(f(x')^2+f(x'')^2)}{\sum_{x\in X(i)}f(x)^2}}\\
&=&\sqrt{\frac{2(\sum_{y\in Y(i)}f(y)^2+\sum_{z\in Z(i)}f(z)^2))}{\sum_{x\in X(i)}f(x)^2}}\\
&\leq &\sqrt{\frac{2\cdot 2\sum_{x\in X(i)}f(x)^2}{\sum_{x\in X(i)}f(x)^2}},\\
&=&2,
\end{eqnarray*}
where the first inequality follows from the Cauchy--Schwarz inequality, the second use the fact $(a+b)^2\leq 2(a^2+b^2)$ and the third follows from Claim 1. 

Combing the above two bounds with (\ref{equ-2}) yields
$$
\lambda\leq 2n-8+2=2n-6,
$$
which implies that there are no eigenvalues in $(2n-6,2n-4)$. Therefore, we have $\lambda_2(AG_n)=2n-6$.
\end{proof}

\section{The second largest eigenvalues of $EAG_n$ and $CAG_n$}\label{s-3}
Recall that the complete alternating group graph is defined as $CAG_n=\mathrm{Cay}(A_n,T_3)$, where  $T_3=\{(i,j,k),(i,k,j)\mid 1\leq i<j<k\leq n\}$  is the set of all $3$-cycles in $S_n$. In this section, we mainly focus on determining the second largest eigenvalue of $CAG_n$. The main idea is to find some suitable equitable partition so that we can decompose the edge set of $CAG_n$ into that of $n$ copies of $CAG_{n-1}$ and that of the extended alternating group graph $EAG_{n}=\mathrm{Cay}(A_n,T_2)$, where $T_2=\{(1,i,j),(1,j,i)\mid 2\leq i<j\leq n\}$ is the set of $3$-cycles in $S_n$ moving $1$, which ensures that the way of induction could be used to find the second largest eigenvalue of $CAG_n$. For this reason, we first need to consider the same problem for the graph $EAG_n$, and again, the main method depends on finding some  suitable equitable partition to decompose the edge set of $EAG_n$. 

For $1\leq i,j\leq n$, we define
\begin{equation}\label{equ-3}
X_i(j)=\{\tau\in A_n\mid \tau_j=i\}.
\end{equation}
For any fixed $i$ (resp. $j$), we see that $X_i(1), X_i(2),\ldots,X_i(n)$ (resp. $X_1(j), X_2(j),\ldots,$ $X_n(j)$) partition $A_n$. Now we shall verify that $\Pi: A_n=X_i(1)\cup X_i(2)\cup \cdots\cup X_i(n)$  is an equitable partition of $EAG_n$ for each $i$.   For $\tau\in X_i(1)$, we have $N(\tau)\cap X_i(1)=\emptyset$ since $(\sigma\tau)_1\neq \tau_1=i$ for any $\sigma\in T_2$, and $N(\tau)\cap X_i(j)=\{(1,k,j)\tau\mid 2\leq k\leq n,k\neq j\}$ for $2\leq j\leq n$. Thus each $\tau\in X_i(1)$ has no neighbors in $X_i(1)$ and $n-2$ neighbors in $X_i(j)$ for each $j$ ($2\leq j\leq n$). For each fixed $j$ ($2\leq j\leq n$), let  $\tau\in X_i(j)$, we have $N(\tau)\cap X_i(1)=\{(1,j,k)\tau\mid 2\leq k\leq n,k\neq j\}$, $N(\tau)\cap X_i(j)=\{(1,k,l)\tau,(1,l,k)\tau\mid 2\leq k<l\leq n,k,l\neq j\}$, and $N(\tau)\cap X_i(\ell)=\{(1,\ell,j)\tau\}$ for any $\ell\not\in\{1,j\}$. This implies that $\tau$ has $n-2$ neighbors in $X_i(1)$, $(n-2)(n-3)$ neighbors in $X_i(j)$, and one neighbor in $X_i(\ell)$ for each $\ell\not\in\{1,j\}$. Thus the partition $\Pi$ is an equitable partition of $EAG_n$ with divisor matrix
$$
B_\Pi=
\left[
\begin{matrix}
0&n-2&n-2&\cdots&n-2\\
n-2&(n-2)(n-3)&1&\cdots&1\\
n-2&1&(n-2)(n-3)&\cdots&1\\
\vdots&\vdots&\vdots&\ddots&\vdots\\
n-2&1&1&\cdots&(n-2)(n-3)
\end{matrix}
\right]
\begin{matrix}
X_i(1)\\
X_i(2)\\
X_i(3)\\
\vdots\\
X_i(n)
\end{matrix}
$$
Then the characteristic polynomial of $B_\Pi$ is equal to
\begin{eqnarray*}
|\lambda I_n-B_\Pi|
&=&
\left|
\begin{smallmatrix}
\lambda &-(n-2)&-(n-2)&\cdots&-(n-2)\\
-(n-2)&\lambda-(n-2)(n-3)&-1&\cdots&-1\\
-(n-2)&-1&\lambda-(n-2)(n-3)&\cdots&-1\\
\vdots&\vdots&\vdots&\ddots&\vdots\\
-(n-2)&-1&-1&\cdots&\lambda-(n-2)(n-3)
\end{smallmatrix}
\right|\\
&=&
(\lambda-(n-1)(n-2))\left|
\begin{smallmatrix}
1 &-(n-2)&-(n-2)&\cdots&-(n-2)\\
1&\lambda-(n-2)(n-3)&-1&\cdots&-1\\
1&-1&\lambda-(n-2)(n-3)&\cdots&-1\\
\vdots&\vdots&\vdots&\ddots&\vdots\\
1&-1&-1&\cdots&\lambda-(n-2)(n-3)
\end{smallmatrix}
\right|\\
&=&
(\lambda-(n-1)(n-2))\left|
\begin{smallmatrix}
1 &0&0&\cdots&0\\
1&\lambda-(n-2)(n-4)&n-3&\cdots&n-3\\
1&n-3&\lambda-(n-2)(n-4)&\cdots&n-3\\
\vdots&\vdots&\vdots&\ddots&\vdots\\
1&n-3&n-3&\cdots&\lambda-(n-2)(n-4)
\end{smallmatrix}
\right|\\
&=&
(\lambda-(n-1)(n-2))\left|
\begin{smallmatrix}
\lambda-(n-2)(n-4)&n-3&\cdots&n-3\\
n-3&\lambda-(n-2)(n-4)&\cdots&n-3\\
\vdots&\vdots&\ddots&\vdots\\
n-3&n-3&\cdots&\lambda-(n-2)(n-4)
\end{smallmatrix}
\right|\\
&=&
(\lambda\!-\!(n-1)(n-2))\left|
(\lambda\!-\!(n-2)(n-4))I_{n-1}\!+\!(n-3)(J_{n-1}\!-\!I_{n-1})
\right|\\
&=&
(\lambda-(n-1)(n-2))(\lambda-(n^2-5n+5))^{n-2}(\lambda+(n-2)),
\end{eqnarray*}
where $I_{n-1}$ and $J_{n-1}$ denote the identity matrix and all ones matrix of order $n-1$, respectively. By Lemma \ref{lem-2-1}, we have 
\begin{lem}\label{lem-3-1}
Let $EAG_n$ be the extended alternating group graph with $n\geq 3$. Then $EAG_n$ has eigenvalues $(n-1)(n-2)$, $n^2-5n+5$ (with multiplicity at least $n-2$) and $2-n$. Moreover, if $f$ is an eigenvector of $EAG_n$ corresponding to any other eigenvalue, then for all $i,j\in\{1,2,\ldots,n\}$, we have $$\sum_{x\in X_i(j)}f(x)=0,$$ where $X_i(j)$ is defined in (\ref{equ-3}). 
\end{lem}
Now we give the second largest eigenvalue of $EAG_n$.
\begin{thm}\label{thm-2}
Let $EAG_n$ be the extended alternating group graph with $n\geq 3$. Then the two largest eigenvalues of $EAG_n$ are $\lambda_1(EAG_n)=(n-1)(n-2)$ and $\lambda_2(EAG_n)=n^2-5n+5$. 
\end{thm}
\begin{proof}
Obviously, $\lambda_1(EAG_n)=(n-1)(n-2)$ because $EAG_n$ is $(n-1)(n-2)$-regular. Now we prove $\lambda_2(EAG_n)=n^2-5n+5$ by induction on $n$. For  $n=3$, we have $EAG_3=K_3$, which gives that $\lambda_2(EAG_3)=-1$, and the result follows. Now suppose $n\geq 4$ and assume that the result holds for $EAG_{n-1}$, i.e., $\lambda_2(EAG_{n-1})=(n-1)^2-5(n-1)+5$. Let $\lambda$ be an eigenvalue of $EAG_n$ that is not equal to $(n-1)(n-2)$, $n^2-5n+5$ or $2-n$, and let $f$ be any fixed eigenvector corresponding to $\lambda$.  Then $f$ must sum to zero on $X_i(j)$ for all $1\leq i,j\leq n$. Now we partition the vertex set of $EAG_n$ as $A_n=X_1(2)\cup X_2(2)\cup \cdots\cup X_n(2)$. As in the proof of Theorem  \ref{thm-1}, one can easily verify that the induce subgraph of $EAG_n$ on $X_i(2)$ is isomorphic to $EAG_{n-1}$ for each $i$ ($1\leq i\leq n$). Let $E_1=\cup_{i=1}^n E(EAG_n[X_i(2)])$ and $E_2=E(EAG_n)\setminus E_1$. We claim that $E_2$ is exactly the set of edges of the alternating group graph $AG_n=\mathrm{Cay}(A_n,T_1)$, where $T_1=\{(1,2,k),(1,k,2)\mid 3\leq k\leq n\}$.  In fact, for any  $\gamma\in X_i(2)$ and $\gamma'\in X_j(2)$ ($i\neq j$), we have $\{\gamma,\gamma'\}\in E(EAG_n)$ if and only if $\gamma'\gamma^{-1}\in T_2$ (i.e., $\gamma'\gamma^{-1}$ is a $3$-cycle moving $1$), which is the case if and only $\gamma'\gamma^{-1}\in T_1(\subseteq T_2)$ because $\gamma'\gamma^{-1}$ must move $2$ due to $i\neq j$. This implies that the edges in $E_2$ come from $T_1$. On the other hand, we see that $T_1$ can only be used to produce the edges in $E_2$ because each edge in $E_1$ comes from $T_2\setminus T_1$, i.e., the set of $3$-cycles in $T_2$ fixing $2$.  Then
\begin{eqnarray*}
\lambda&=&\frac{f^TA(EAG_n)f}{f^Tf}\\
&=&\frac{2\sum_{\{x,y\}\in E(EAG_n)}f(x)f(y)}{\sum_{x\in A_n}f(x)^2}\\
&=&\frac{2\sum_{\{x,y\}\in E_1}f(x)f(y)}{\sum_{x\in A_n}f(x)^2}+\frac{2\sum_{\{x,y\}\in E_2}f(x)f(y)}{\sum_{x\in A_n}f(x)^2}.
\end{eqnarray*}
For the first term, we have
\begin{eqnarray*}
\frac{2\sum_{\{x,y\}\in E_1}f(x)f(y)}{\sum_{x\in A_n}f(x)^2}&=&\frac{\sum_{i=1}^n2\sum_{\{x,y\}\in E(EAG_n[X_i(2)])}f(x)f(y)}{\sum_{i=1}^n\sum_{x\in X_i(2)}f(x)^2}\\
&\leq &\max_{1\leq i\leq n}\frac{2\sum_{\{x,y\}\in E(EAG_n[X_i(2)])}f(x)f(y)}{\sum_{x\in X_i(2)}f(x)^2}\\
&\leq &\lambda_2(EAG_{n-1}),
\end{eqnarray*}
where the last inequality follows from the fact $\sum_{x\in X_i(2)}f(x)=0$ for each $i$ according to Lemma \ref{lem-3-1}. For the second term, since $f$ is orthogonal to the all ones vector $\mathbf{1}$, we have 
$$
\frac{2\sum_{\{x,y\}\in E_2}f(x)f(y)}{\sum_{x\in A_n}f(x)^2}\leq\max_{g\perp\mathbf{1}}\frac{g^TA(AG_n)g}{g^Tg}=\lambda_2(AG_n).
$$
Combining above two bounds, we conclude that 
$$\lambda\leq \lambda_2(EAG_{n-1})+\lambda_2(AG_n)=(n-1)^2-5(n-1)+5+2n-6=n^2-5n+5$$
by Theorem \ref{thm-1}. Hence $\lambda_2(EAG_n)=n^2-5n+5$, and our result follows.
\end{proof}

Now we focus on the complete alternating group graph $CAG_n=\mathrm{Cay}(A_n,T_3)$ where $T_3=\{(i,j,k),(i,k,j)\mid 1\leq i<j<k\leq n\}$. Obviously, $CAG_n$ is a $2\binom{n}{3}$-regular graph. As above, we shall verify that $\Pi: A_n=X_i(1)\cup X_i(2)\cup \cdots\cup X_i(n)$ is also an equitable partition of $CAG_n$ for each fixed $i$, where $X_i(j)$ is defined in  (\ref{equ-3}).  For each fixed $j$ ($1\leq j\leq n$), let $\tau\in X_i(j)$. We have $N(\tau)\cap X_i(j)=\{(k,l,m)\tau,(k,m,l)\tau\mid1\leq k<l<m\leq n, k,l,m\neq j\}$, and $N(\tau)\cap X_i(\ell)=\{(\ell,j,k)\tau\mid 1\leq k\leq n, k\neq\ell, j\}$ for each $\ell\neq j$. This implies that $\tau$ has $2\binom{n-1}{3}$ neighbors in $X_i(j)$, and $n-2$ neighbors in $X_i(\ell)$ for each $\ell\neq j$. Thus $\Pi$ is an equitable partition of $CAG_n$ with divisor matrix
$$
B_\Pi=
\left[
\begin{matrix}
2\binom{n-1}{3}&n-2&n-2&\cdots&n-2\\
n-2&2\binom{n-1}{3}&n-2&\cdots&n-2\\
n-2&n-2&2\binom{n-1}{3}&\cdots&n-2\\
\vdots&\vdots&\vdots&\ddots&\vdots\\
n-2&n-2&n-2&\cdots&2\binom{n-1}{3}
\end{matrix}
\right]
\begin{matrix}
X_i(1)\\
X_i(2)\\
X_i(3)\\
\vdots\\
X_i(n)
\end{matrix}.
$$
It is easy to see that $B_\Pi$ has eigenvalues $2\binom{n}{3}$ and $\frac{1}{3}n(n-2)(n-4)$, where the first one is of  multiplicity $1$ having the all ones vector $\mathbf{1}$ as its eigenvector, and the second one is of multiplicity $n-1$ having each vector orthogonal to $\mathbf{1}$ as an eigenvector. Again by Lemma \ref{lem-2-1}, we have
\begin{lem}\label{lem-4-1}
Let $CAG_n$ be the complete alternating group graph with $n\geq 3$. Then $CAG_n$ has eigenvalues $2\binom{n}{3}$ and $\frac{1}{3}n(n-2)(n-4)$ (with multiplicity at least $n-1$). Moreover, if $f$ is an eigenvector of $CAG_n$ corresponding to any other eigenvalue, then for all $i,j\in\{1,2,\ldots,n\}$, we have $$\sum_{x\in X_i(j)}f(x)=0,$$ where $X_i(j)$ is defined in (\ref{equ-3}). 
\end{lem}
Now we prove the main result of this section.
\begin{thm}\label{thm-3}
Let $CAG_n$ be the complete alternating group graph with $n\geq 3$. Then the two largest eigenvalues of $CAG_n$ are $\lambda_1(CAG_n)=2\binom{n}{3}$ and $\lambda_2(CAG_n)=\frac{1}{3}n(n-2)(n-4)$. 
\end{thm}
\begin{proof}
Clearly, we have $\lambda_1(CAG_n)=2\binom{n}{3}$ because $CAG_n$ is $2\binom{n}{3}$-regular. We shall prove $\lambda_2(CAG_n)=\frac{1}{3}n(n-2)(n-4)$ by induction on $n$. For $n=3$, we have $\lambda_2(CAG_3)=\lambda_2(K_3)=-1$, as required. Now suppose $n\geq 4$. Assume that the result holds for $n-1$, i.e., $\lambda_2(CAG_{n-1})=\frac{1}{3}(n-1)(n-3)(n-5)$. Let $\lambda$ be any eigenvalue of $CAG_n$ that is not equal to $2\binom{n}{3}$ or $\frac{1}{3}n(n-2)(n-4)$, and pick any eigenvector $f$ of $CAG_n$ corresponding to $\lambda$. Again, the vector $f$ must sum to zero on $X_i(j)$ for all $i,j$ according to Lemma \ref{lem-4-1}. We partition the vertex set of $CAG_n$ as $A_n=X_1(1)\cup X_2(1)\cup \cdots\cup X_n(1)$. Let $E_1=\{\{\tau,\sigma\}\in E(CAG_n)\mid \tau_1\neq \sigma_1\}$ and $E_2^i=\{\{\tau,\sigma\}\in E(CAG_n)\mid \tau_1=\sigma_1=i\}$ for $1\leq i\leq n$. Then $E_1\cup E_2^1\cup E_2^2\cup \cdots\cup E_2^n$ is a partition of $E(CAG_n)$. As in the proof of Theorem \ref{thm-2}, we see that  $E_1$ is exactly the set of edges of the extended alternating group graph $EAG_n=\mathrm{Cay}(A_n,T_2)$ with $T_2=\{(1,k,l),(1,l,k)\mid 2\leq k<l\leq n\}$, and for each $i$, $E_2^i$ is exactly the set of edges of the induced subgraph of $CAG_n$ on $X_i(1)$ which is also isomorphic to  $CAG_{n-1}$. Then we have
\begin{eqnarray*}
\lambda&=&\frac{f^TA(CAG_n)f}{f^Tf}\\
&=&\frac{2\sum_{\{x,y\}\in E(CAG_n)}f(x)f(y)}{\sum_{x\in A_n}f(x)^2}\\
&=&\frac{2\sum_{\{x,y\}\in E_1}f(x)f(y)}{\sum_{x\in A_n}f(x)^2}+\frac{\sum_{i=1}^n2\sum_{\{x,y\}\in E_2^i}f(x)f(y)}{\sum_{x\in A_n}f(x)^2}.
\end{eqnarray*}
For the first term, since $f$ is orthogonal to the all ones vector $\mathbf{1}$, we have 
$$
\frac{2\sum_{\{x,y\}\in E_1}f(x)f(y)}{\sum_{x\in A_n}f(x)^2}\leq\max_{g\perp\mathbf{1}}\frac{g^TA(EAG_n)g}{g^Tg}=\lambda_2(EAG_n).
$$
For the second term, we have
\begin{eqnarray*}
\frac{\sum_{i=1}^n2\sum_{\{x,y\}\in E_2^i}f(x)f(y)}{\sum_{x\in A_n}f(x)^2}&=&\frac{\sum_{i=1}^n2\sum_{\{x,y\}\in E(CAG_n[X_i(1)])}f(x)f(y)}{\sum_{i=1}^n\sum_{x\in X_i(1)}f(x)^2}\\
&\leq &\max_{1\leq i\leq n}\frac{2\sum_{\{x,y\}\in E(CAG_n[X_i(1)])}f(x)f(y)}{\sum_{x\in X_i(1)}f(x)^2}\\
&\leq &\lambda_2(CAG_{n-1}),
\end{eqnarray*}
where the last inequality follows from the fact $\sum_{x\in X_i(1)}f(x)=0$ for each $i$ according to Lemma \ref{lem-3-1}.  Combining above two bounds, we conclude that 
$$\lambda\leq \lambda_2(EAG_{n})+\lambda_2(CAG_{n-1})=n^2-5n+5+\frac{1}{3}(n-1)(n-3)(n-5)=\frac{1}{3}n(n-2)(n-4)$$
by Theorem \ref{thm-2}. Hence $\lambda_2(EAG_n)=\frac{1}{3}n(n-2)(n-4)$, as required.
\end{proof}

At the end of this paper, we pose the following problem.

\begin{prob}\label{prob-1}
How to determine all distinct eigenvalues (with multiplicities) of $AG_n$, $EAG_n$ and $CAG_n$?
\end{prob}

\section*{Acknowledgments}
The authors are grateful to the anonymous referees for their useful and constructive comments, which have considerably improved the presentation of this paper.

\end{document}